\documentclass[a4paper,12pt]{article}

\usepackage[utf8]{inputenc}
\usepackage{a4}
\usepackage{amssymb,amsmath,amsfonts,amsthm}
\usepackage[matrix,arrow,curve]{xy} 
\usepackage{verbatim,url}

\usepackage{hyperref}

\usepackage{graphicx}
\usepackage{longtable}

\newtheorem{theorem}{Theorem}[section]
\newtheorem{proposition}[theorem]{Proposition}

\newtheorem{lemma}[theorem]{Lemma}

\newtheorem{definition}[theorem]{Definition}

\newtheorem{remark}[theorem]{Remark}

\newtheorem{thm}[theorem]{Theorem}

\newcommand{\si}{\sigma}
\newcommand{\OH}{\mathcal O}

\title{The Classification of Regular Surfaces Isogenous to a Product of Curves with $\chi(\OH_S)= 2$}
\date{}

\author{Christian Glei\ss ner}

\begin{document}

\maketitle

\bigskip

\section{Introduction}

\bigskip

A complex surface  $S$ is said to be \emph{isogenous to a product} if $S$ is a quotient 
\[
S=(C_1 \times C_2)/G,
\]
where the $C_i$'s  are curves of genus
at least two, and $G$ is a finite group acting freely on $C_1 \times C_2$. 
Due to Catanese \cite{cat} there are two possibilities how $G$ can act on the product $ C_1 \times C_2 $:
\begin{itemize}
\item For all $g \in G$ and $(z,w) \in C_1 \times C_2$ we have $g(z,w)=(g(z),g(w))$.  
In this case the action is called \emph{diagonal}.
\item There exists  $g \in G $, such that  $g(z,w)=(g(w),g(z))$
for all  $(z,w) \in C_1 \times C_2$. In this case the curves  $C_1$ and $C_2$ are isomorphic.
\end{itemize} 
\noindent
If the action of $G$ is diagonal,  we say  that $S$ is of \emph{unmixed type}, else we say that $S$ is of  \emph{mixed type}. \\
Since these surfaces were introduced by Catanese \cite{cat} there has been produced a considerable amount of literature. 
In particular the surfaces isogenous to a product with $\chi(\OH_S)= 1$ are completely classified:
The Bogomolov-Miyaoka-Yau inequality $K_S^2 \leq 9 \chi(\mathcal O_S)$ together with 
Debarre's inequality $K^2_S\geq 2p_g(S) $ gives $0\leq q(S) = p_g(S) \leq 4$.
By Beauville \cite{Be82} all minimal surfaces $S$ of general type with $p_g(S)=q(S)=4$  are a product of two genus two curves. A minimal surface $S$ of general type with $p_g(S)=q(S)=3$ is either the symmetric square of a genus three curve or 
$S=(F_2 \times F_3)/\tau$, where $F_g$ is a curve of genus $g$ and $\tau$  is of order two  acting on $F_2$ as an elliptic involution and on 
$F_3$ as a fixed point free involution \cite{CCML98, Pir02, HP02}. The classifications of surfaces isogenous to a product 
in the remaining cases are: 
the case $p_g=0$, $q=0$, due to  Bauer, Catanese, Grunewald \cite{bauer}, $p_g=1$, $q=1$, due to Carnovale, Polizzi \cite{polizzi}
and 
$p_g=2$, $q=2$, due to Penegini  \cite{pen10}.

\bigskip
Our aim is to give a classification in the case $\chi(\OH_S)= 2$ under the assumption that $S$ regular and of unmixed type.
We want to mention that these surfaces have the invariants $q(S)=0$ and $p_g(S)=1$  like  $K3$ surfaces and there are recent 
constructions of $K3$ surfaces with non-symplectic automorphisms as \emph{product quotient surfaces} by 
Garbagnati and Penegini \cite{GP13}. Our main result is the following (see also the table in \ref{Haupterg}):

\begin{thm}\label{families} 
There are exactly $49$ families of regular surfaces isogenous to a product of unmixed type with  $\chi(\OH_S)= 2$. 
\end{thm}

We will now explain how the paper is organized. 
In section \ref{section2} we explain the basics about surfaces isogenous to a product of curves. 
In section \ref{section3} we recall \emph{Riemann's existence theorem} and introduce
the necessary tools from group theory and combinatorics. These facts are used to show that there is an entirely group theoretic description of surfaces isogenous to a product. In section \ref{section4} we recall the description of the moduli space of surfaces isogenous to a product, due to Catanese.  
In section \ref{section5} we give an algorithm, which we use to classify all regular surfaces isogenous to a product of 
unmixed type with $\chi(\OH_S)=2$. 
The computations are performed with the computer algebra system MAGMA \cite{magma}. \footnote{The source code is available at
\url{http://www.staff.uni-bayreuth.de/~ bt300503/.}}
In particular the \emph{Database of Small Groups} and the \emph{Database of Perfect Groups} is used. Afterwards we discuss the output of the 
computation. Finally we show the classification result.  

\bigskip

\textbf{Acknowledgments:} The autor thanks Ingrid Bauer and Sascha Weigl for several
suggestions, useful discussions and very careful reading of the paper.

\bigskip

\section{Surfaces Isogenous to a Product}\label{section2}

\bigskip
In this section we explain some basic facts about surfaces isogenous to a product. 
We work over the field of complex numbers $\mathbb C$ and use the standard notations from the theory of complex surfaces,
see for example \cite{Be83}.
The \emph{self-intersection of the  canonical class} is denoted by  $K_S^2$, the \emph{topological Euler number} by 
$e(S)=\sum_{i=1}^{4} (-1)^{i} h^i(S,\mathbb C)$, $\kappa(S)$ is the Kodaira dimension. The \emph{holomorphic-Euler-Poincar\'e-characteristic} is defined as $\chi(\OH_S):= 1 - q(S) + p_g(S)$, where $q(S):=h^1(S,\OH_S)$ and $p_g(S):=h^2(S,\OH_S)$. 
The basic objects we consider are the following:

\newpage

\begin{definition}

A surface  $S$ is said to be \emph{isogenous to a product} if $S$ is a quotient 
\[
S=(C_1 \times C_2)/G,
\]
where the $C_i$'s  are smooth projective curves of genus
at least two, and \\ $G \leq Aut(C_1 \times C_2)$  is a finite group of automorphisms acting \emph{freely} on $C_1 \times C_2$.
\end{definition}

\bigskip
Immediate consequences of this definition are:  
The surface $S$ is smooth, projective and of general type, i.e. $\kappa(S)=2$.
The canonical class $K_S$ is ample. In particular $S$ is minimal.

\bigskip
The \emph{self-intersection of the  canonical class} $K_S^2$, the \emph{topological Euler number} and the 
\emph{holomorphic-Euler-Poincar\'e-characteristic} can be expressed in terms of the genera $g(C_i)$:  
\bigskip

\begin{proposition}\cite[Theorem 3.4]{cat} 
Let $S = (C_1 \times C_2)/G $ be a surface isogenous to a product, then 
\[
K_S^2=\frac{8(g(C_1) -1)(g(C_2) -1)}{|G|}, ~~~ 
e(S)= \frac{1}{2} K_S^2 ~~~  \makebox{and} ~~~ \chi(\OH_S)= \frac{1}{8} K_S^2. 
\]
\end{proposition}
In our case $p_g(S)=1$ and $q(S)=0$ we get  $K_S^2=16$ and $e(S)=8$, moreover we have the useful relation
\begin{gather}
2 |G| = (g(C_1) -1)(g(C_2) -1).
\label{1}
\end{gather} 
\bigskip

\begin{remark}
For the rest of the paper we consider the \emph{unmixed} case, where the action of $G$ on the product 
$C_1 \times C_2$ is diagonal, i.e. 
\[
G = G \cap (Aut(C_1) \times Aut(C_2)). 
\]
\end{remark}

Since we consider unmixed actions only, we obtain two Galois coverings  
\[
f_1: C_1 \to C_1/G, ~~~ f_2: C_2 \to C_2/G.  
\]
By \cite[Proposition 3.13]{cat} one can assume without loss of generality, that $G$ acts faithfully on $C_1$ and on $C_2$. 
We want to relate the invariants $p_g(S)$ and $q(S)$ with the genera of the curves. 
To do this, we need the following theorem. 

\bigskip

\begin{theorem}
Let $X$ be a smooth projective variety  and $G$ be a finite group acting faithfully on $X$. 
If $Y=X/G$ is smooth, then 
\[
H^0(Y,\Omega_Y^p) \simeq H^0(X,\Omega_X^p)^G. 
\]
\end{theorem}
A proof of this result can be found in \cite{griffiths}.

\bigskip
By K\"unneth's formula \cite[p.103-104]{GriffH}:
\[
H^0(C_1 \times C_2 ,\Omega_{C_1 \times C_2}^1)^G= H^0(C_1,\Omega_{C_1}^1)^G \oplus H^0(C_2,\Omega_{C_2}^1)^G.
\]
According to the previous theorem $q(S)=h^0(C_1 \times C_2 ,\Omega_{C_1 \times C_2}^1)^G$. Since $q(S)=0$ by assumption, we conclude that $g(C_i/G)=0$ for both $i=1,2$. Thus the holomorphic maps $f_i$ from above are Galois coverings of $\mathbb P_{\mathbb C}^1$.

\bigskip

\section{Group theory, Riemann surfaces and combinatorics}\label{section3}

\bigskip
\bigskip
To give a purely group theoretic description of surfaces isogenous to a product, 
we introduce the required notation from group theory and combinatorics and recall \emph{Riemann's existence theorem}.

\bigskip

\begin{definition}\label{condtypes}

Let  $ T = [m_1,...,m_r] \in \mathbb N^r $ be an $r$-tuple. We define 
\[
\Theta(T):= -2 + \sum_{i=1}^r \bigg( 1- \frac{1} {m_i} \bigg)
\]
and in case $\Theta(T) \ne 0$ 
\[
\alpha(T):= \frac{4} {\Theta(T)}.
\]
For $r \geq 3$ we denote by  $\mathcal N_r$ the set of all  r-tuples $[m_1,...,m_r]$ with the following properties:
\begin{itemize}
\item $ 2 \leq m_1 \leq ... \leq m_r $ 
\item $ \Theta([m_1, ... ,m_r]) > 0 $ 
\item $ \alpha([m_1, ... ,m_r]) \in \mathbb N $ 
\item $ m_i  \big\vert  \alpha([m_1, ... ,m_r]) ~~  \makebox{for all} ~~  1 \leq i \leq r $.  
\end{itemize}
The union of  all $\mathcal N_r$ is defined as 
$
\mathcal N := \bigcup_{r \geq 3}{\mathcal N_r}.
$
An element in $\mathcal N$ is called a type. A type $T$ contained in $\mathcal N_r$ is said to be of length $r$, we write $\it l(T)=r$. 
Moreover for a type $T=[m_1,...,m_r]$ we use the notation $T=[m_1,...,m_r]_{\alpha(T)}$.
\end{definition}

\bigskip
For simplicity we write
\[
[a_1^{k_1},...,a_r^{k_r}]:=[\underbrace {a_1,...,a_1}_{\makebox{$k_1$-times}},...,\underbrace {a_r,...,a_r}_{\makebox{$k_r$-times}}].
\]
In the following lemma we give a classification of all types which satisfy the conditions from definition \ref{condtypes} above.
This is the starting point of the classification of regular surfaces isogenous to a product with $\chi(\mathcal O_S)=2$ (see also \ref{condisogenous}). 

\bigskip

\begin{lemma}\label{tupel}
There are no types of length $r$ if  $r=7$ or $r \geq 9$.  The set $\mathcal N$ is finite and given by:

\[
\mathcal N = \left\lbrace
   \begin{array}{lllll}
\mbox{$[2,3,7]_{168}$},  & \mbox{$[2,3,8]_{96}$},   &  \mbox{$[2,4,5]_{80}$},   &  \mbox{$[2,3,9]_{72}$},  &  \mbox{$[2,3,10]_{60}$},  \\	  
\mbox{$	[2,3,12]_{48}$}, & \mbox{$[2,4,6]_{48}$},   &  \mbox{$[3^2,4]_{48}$},   &  \mbox{$[2,3,14]_{42}$}, &  \mbox{$[2,5^2]_{40}$},  \\  
\mbox{$[2,3,18]_{36}$},  & \mbox{$	[2,4,8]_{32}$},  & \mbox{$[2,3,30]_{30}$},   &  \mbox{$[2,5,6]_{30}$},  &  \mbox{$[3^2,5]_{30}$},  \\
\mbox{$[2,4,12]_{24}$},  & \mbox{$[2,6^2]_{24}$},   &  \mbox{$[3,4^2]_{24}$},   &  \mbox{$[3^2,6]_{24}$},  &  \mbox{$	[2^3,3]_{24}$}, \\
\mbox{$[3^2,7]_{21}$},   & \mbox{$[2,4,20]_{20}$},  &  \mbox{$[2,5,10]_{20}$},  &  \mbox{$[2,6,9]_{18}$},  &  \mbox{$[3^2,9]_{18}$},  \\
\mbox{$[2,8^2]_{16}$},   & \mbox{$[4^3]_{16}$},     &  \mbox{$[2^3,4]_{16}$},   &   \mbox{$[3,5^2]_{15}$}, &  \mbox{$[3^2,15]_{15}$}, \\                  \mbox{$[2,7,14]_{14}$},  &  \mbox{$[2,12^2]_{12}$}, &  \mbox{$[3,4,12]_{12}$},  &  \mbox{$[3,6^2]_{12}$},  &  \mbox{$[4^2,6]_{12}$},   \\
\mbox{$[2^3,6]_{12}$},   &  \mbox{$[2^2,3^2]_{12}$}, &  \mbox{$[5^3]_{10}$},    &  \mbox{$[2^3,10]_{10}$}, &  \mbox{$ [3,9^2]_{9}$}, \\  
\mbox{$[4,8^2]_{8}$},    &  \mbox{$[2^2,4^2]_{8}$},  &  \mbox{$[2^5]_{8}$},     &  \mbox{$[7^3]_{7}$},     &	\mbox{$[2^2,6^2]_{6}$},   \\  
\mbox{$[2,3^2,6]_{6}$},  &  \mbox{$[3^4]_{6}$},      &  \mbox{$[2^4,3]_{6}$},   &  \mbox{$[4^4]_{4}$},     &	\mbox{$[2^3,4^2]_{4}$},   \\  
\mbox{$[2^6]_{4}$},      &  \mbox{$[3^5]_{3}$},      &  \mbox{$[2^8]_{2}$}  	 
\end{array}
\right\rbrace
\]

\end{lemma}

\bigskip

\begin{proof}
We use the fourth property in the case $i=r$: 
\[
\makebox{From} ~~ m_r \big\vert   \frac{4}{-2 + \sum_{i=1}^r (1- \frac{1} {m_i})} ~~ \makebox{it follows} ~~ 
\sum_{i=1}^r (1- \frac{1} {m_i}) \leq 2 + \frac{4} {m_r} ~~ (\ast).
\]
Since $ m_i \geq 2 $ for all  $ 1 \leq i \leq r $, we get 
\[
r-2 \leq \sum_{i=1}^r  \frac{1} {m_i} + \frac{4} {m_r} \leq \frac{r} {2} + 2,
\]
and therefore $r \leq 8$.
\noindent
We now investigate two cases: $r =3$ and $r \geq 4 $. 
\begin{itemize}
\item If $r =3$, we claim that $ m_2 \geq 3 $. Suppose $ m_2 = 2 $, then also $ m_1 = 2 $ and  
\[
0 < \Theta([m_1,m_2,m_3])= 1 -  \frac{1} {m_1} -  \frac{1} {m_2}  -  \frac{1} {m_3} = -  \frac{1} {m_3},  
\]
a contradiction.
The inequality $(\ast)$ in the case  $r=3$ reads
\[
\frac{5} {m_3} \geq 1 - \frac{1} {m_1} - \frac{1} {m_2} \geq 1 - \frac{1} {2} - \frac{1} {3} = \frac{1} {6}.  
\]
From this we conclude  $m_3 \leq 30$. 
\item
In the second case $r \geq 4$, we use the formula $(\ast)$ again:
\[
r - \sum_{i=1}^r\frac{1} {m_i} \leq 2 + \frac{4} {m_r}
\]
and get
\[
\frac{5} {m_r} \geq r - 2 - \sum_{i=1}^{r-1}\frac{1} {m_i} \geq \frac{r - 3}{2}.
\]
Because of this, we always have  $10 \geq  \frac{10} {r - 3} \geq m_r $ in that case. 
\end{itemize}
Now, it suffices to check only a finite number of types. This can be easily done with a computer.
\end{proof}

\bigskip

\begin{definition}
Let $G$ be a finite group, $2 \leq m_1 \leq ... \leq m_r$ integers. 
A \emph{spherical system} of generators of $G$  of type $[m_1, ... ,m_r]$ is an $r$-tuple 
$A=(g_1,...,g_r)$ of elements of $G$, such that: 
\begin{itemize}
\item
$G = \langle g_1, ... ,g_r \rangle$, ~~ $g_1 \cdot ... \cdot g_r= 1_G$.
\item
There exists a permutation $ \tau \in \mathfrak S_r$, such that  $ord(g_i)= m_{\tau(i)}$.
\end{itemize}
The \emph{stabilizer set} of  $A$ is defined as  
\[
\Sigma(A):= \bigcup_{h \in G} \bigcup_{j \in \mathbb Z} \bigcup_{i=1}^r \{hg_i^jh^{-1}\}.
\]
A pair $(A_1,A_2)$ of spherical systems of generators  of  $G$ is called  \emph{disjoint}, if and only if
\[
\Sigma(A_1) \cap \Sigma(A_2)= \{1_G\}.
\] 
\end{definition}

\bigskip

The geometry behind this definition is known as \emph{Riemann's existence theorem}. A detailed explanation can be found in 
\cite[chapter III, sections 3 and 4]{mir}. We will use the following version of this theorem:  

\bigskip

\begin{theorem}[Riemann's existence theorem]
A finite group $G$ acts as a group of automorphisms on a compact Riemann surface $C$ of genus $g(C) \geq 2$, such that  
$C/G \simeq \mathbb P_{\mathbb C}^1$ 
if and only if  there exists a spherical system of generators $A$ of $G$ of type $T=[m_1,...,m_r]$, such that
the following Riemann-Hurwitz formula holds:
\[
2g(C) -2 = |G|\Theta(T).
\]
\end{theorem}

\bigskip

By Riemann's existence theorem we have a  \emph{group theoretical description} of surfaces isogenous to a product: 
Given 
$S=(C_1 \times C_2) / G$, isogenous to a product, we can attach 
a disjoint pair of spherical systems of generators 
\[
(A_1(S),A_2(S)) ~~ \makebox{of type} ~~  (T_1(S),T_2(S)).
\]
Geometrically, disjoint means that $G$ acts without fixed points on $C_1 \times C_2$.  
Conversely, the data above determine a surface isogenous to a product.

Next, we want to show that the types $(T_1(S),T_2(S))$, attached to a regular surface $S$ isogenous to a product with $p_g(S)=1$ of unmixed type, satisfy the conditions of \ref{condtypes}. The proof of this fact is similar to the proof given in \cite{bauer}. For convenience of the reader we will present the proof. 

\bigskip

\begin{theorem}\label{condisogenous} 

Let $S$ be a surface isogenous to a product of curves of unmixed type with $ p_g(S)=1$ and $q(S)=0 $.
Let  $ T_1(S)=[m_1, ... ,m_r]$ and $T_2(S)= [n_1, ... ,n_s] $ be the corresponding types, then   

\begin{itemize}
\item $ \Theta(T_i(S)) > 0$ for $i=1,2$.
\item $ \alpha(T_i(S)) \in  \mathbb N$ for $i=1,2$.
\item $ m_i \big\vert \alpha(T_1(S)) $  for all $ 1 \leq i \leq r $  and   $n_i \big\vert \alpha(T_2(S))$  for all  $ 1  \leq i \leq s $. 

\end{itemize}
\end{theorem} 

\bigskip

\begin{proof} 
We consider the holomorphic maps $f_i: C_i \to C_i/G$
and apply the Riemann-Hurwitz formula 
\begin{gather}
2g(C_i) - 2 = |G|  \Theta(T_i(S)),  \quad i=1,2. 
\label{2}
\end{gather} 
This already shows the first claim, since $g(C_1) \geq 2 $ and $g(C_2) \geq 2 $. 
From  (\ref{2}) and  $ 2 |G|= (g(C_1) -1 )(g(C_2) -1 )$ (\ref{1}), we deduce  
\[ 
\alpha(T_1(S)) = \frac{4}{\Theta(T_1(S))}= g(C_2)-1  ~~ \makebox{and} ~~  
\alpha(T_2(S)) =\frac{4}{\Theta(T_2(S))} = g(C_1)-1 
\] 
so the second claim follows. 
\noindent
It remains to prove the third claim. Let $ A_1(S)=(g_1, ... ,g_r) $ be a corresponding ordered spherical system of generators of  $G$ of type $T_1(S)$.
The cyclic group  $\langle g_i \rangle $ of order $m_i$ acts on  $C_1$ with at least one fixed point, but the action on the product  $C_1 \times C_2$ 
is free. Therefore $\langle g_i \rangle $ acts on $C_2 $ freely. The map $C_2 \to C$ 
of degree $m_i$, where $C:=C_2/{\langle g_i \rangle }$ is unramified. In this case we have $2g(C)-2= \frac{2g(C_2)-2}{m_i} > 0$, due to Riemann-Hurwitz. Hence 
\[
g(C_2)-1 = \alpha(T_1(S)) = m_i (g(C)-1),
\]
for all  $i=1, ... ,r$. With the same argument we can show that  $n_i  \big\vert \alpha(T_2(S))$ for all $i=1, ... ,s$.  
\end{proof}

\bigskip

\section{Moduli Spaces}\label{section4}

\bigskip
\bigskip

In this section we want to describe the moduli space of surfaces isogenous to a product. We follow the papers \cite[S31,S8-9]{bauer} and
\cite[Appendix]{pen11}.
Due to the work of Gieseker \cite{Gieseker} there exists a quasi-projective moduli space of minimal smooth projective surfaces of general type 
with fixed invariants  $K_S^2$ and $\chi( \OH_S)$, which is denoted by $ \mathfrak M_{(\chi(\OH_S),K_S^2)}$.
For a fixed finite group $G$ and a fixed pair of types $(T_1,T_2)$ we denote the subset of  $\mathfrak M_{(2,16)}$ of isomorphism
classes of surfaces isogenous to a product, which admit a disjoint pair of spherical systems of generators $(A_1,A_2)$ of type 
$(T_1,T_2)$, by $\mathfrak M_{(G,T_1,T_2)}$.

\bigskip

\begin{theorem}\cite[Remark 5.1]{bauer}
\begin{itemize}
\item
The subset $\mathfrak M_{(G,T_1,T_2)} \subset \mathfrak M_{(2,16)}$  consists of a finite number of
connected components of the same dimension, which are irreducible in the
Zariski topology.
\item
The dimension $d(G,T_1,T_2)$ of any component in $\mathfrak M_{(G,T_1,T_2)}$ 
is 
\[
d(G,T_1,T_2)= \it l (T_1) -  3 + \it l (T_2) -  3.
\]
\end{itemize}
\end{theorem}

\bigskip
The problem to determine  the number $n$ of the connected components of $\mathfrak M_{(G,T_1,T_2)}$ can be translated in a group theoretical problem. We recall the following definition:

\bigskip

\begin{definition}

Let $r \in \mathbb N$ be a positive integer. We define the \emph{Artin-Braid group} $\boldsymbol B_r$ as 
\[
\boldsymbol B_r :=
\left \langle \si_1, ... ,\si_{r-1} \left| \left. \begin{matrix} \si_i \si_j = \si_j \si_i \quad if \quad |i-j| > 1\\ 
                                                 \si_i \si_{i+1} \si_i = \si_{i+1} \si_i \si_{i+1}   
                                                 \quad for \quad i=1, ... ,r-2
                                                           
\end{matrix} \right\rangle \right. \right..
\]
\end{definition}

\bigskip

Let $G$ be a finite group and $T$ be a type of length $\it l(T)=r$. We denote the set of spherical systems of generators of $G$ of type $T$ by $\mathcal B(G,T)$. 
The Artin-Braid group $\boldsymbol B_r$ acts on $\mathcal B(G,T)$ as follows:  
\[
\si_i(A):=(g_1, ... ,g_{i-1},g_i \cdot g_{i+1} \cdot g_{i}^{-1},g_i,g_{i+2}, ... ,g_r),
\]
for all $A=(g_1, ... ,g_r) \in \mathcal B(G,T)$ and $1 \leq i \leq r-1$. 
This determines a well-defined action, which is called the \emph{Hurwitz action}. There is also a natural action of 
$Aut(G)$ on $\mathcal B(G,T)$: 
\[
\varphi(A):=(\varphi(g_1), ... ,\varphi(g_r)),
\]
for all $\varphi \in Aut(G)$.
Let  $(\gamma_1,\gamma_2,\varphi) \in \boldsymbol B_r \times \boldsymbol B_s \times Aut(G)$ be a triple and $(A_1,A_2) \in \mathcal B(G,T_1)\times \mathcal B(G,T_2)$, we define 
\[
(\gamma_1,\gamma_2,\varphi) \cdot (A_1,A_2) :=(\varphi(\gamma_1(A_1)),\varphi(\gamma_2(A_2))). 
\]
It is easy to verify that this defines an action of $\boldsymbol B_r \times \boldsymbol B_s \times Aut(G)$ on \\ 
$ \mathcal B(G,T_1) \times \mathcal B(G,T_2)$. We denote this action by $\mathfrak T$, and its restriction to the first factor 
$ \mathcal B(G,T_1)$ by $\mathfrak T_1$ and to the second factor $ \mathcal B(G,T_2)$ by $\mathfrak T_2$. 
\footnote{We want to stress that the action of $Aut(G)$ and the Hurwitz action commute.}

\bigskip

\begin{theorem}\cite[Proposition 5.2]{bauer}
Let  $S$ and $S'$ be surfaces isogenous to a product of unmixed
type with $q(S)=q(S')=0$. The surfaces $S$ and $S'$ are in the same irreducible
component if and only if
\begin{itemize}
\item $G(S) \simeq G(S')$,
\item $(T_1(S),T_2(S))=(T_1(S'),T_2(S'))$,
\item $(A_1(S),A_2(S))$ and  $(A_1(S'),A_2(S'))$  are in the same $\mathfrak T$-orbit, or \\
$(A_1(S),A_2(S))$ and  $(A_2(S'),A_1(S'))$ are in the same $\mathfrak T$-orbit.

\end{itemize}
\end{theorem}

\bigskip
In theory we now have a  method to compute the number of connected components of $\mathfrak M_{(G;T_1,T_2)}$:
first we compute a representative $(A_1,A_2)$ for each orbit of the action $\mathfrak T$. 
Then we determine the pairs, where the intersection $\Sigma(A_1) \cap \Sigma(A_2)$ is trivial. If $T_1(S)=T_2(S)$ and there is more than one pair, we have to consider the $\mathbb Z_2$ action corresponding to the exchange of the curves.\footnote{This happens in only one of our examples. Thus we consider only the $\mathfrak T$ action in our program and treat the exceptional example separately. \ref{exchange}} 
The number of the remaining pairs is the number of connected components of $\mathfrak M_{(G;T_1,T_2)}$.
However, the set $ \mathcal B(G,T_1) \times \mathcal B(G,T_2)$ can be very large. Even with a computer it is not possible, or at least very time consuming, to perform this calculation. 
To improve the speed of the calculation we use an idea of Penegini and Rollenske, which is based on the following lemma:

\bigskip

\begin{lemma}\cite[Appendix 6.1]{pen11}\label{pen11}
Let $(A_1,A_2)$, $(B_1,B_2) \in \mathcal B(G,T_1)\times \mathcal B(G,T_2)$. 
\begin{itemize}

\item If  $A_i$ and $B_i$  are in the same orbit of the Hurwitz action, then also the pairs $(A_1,A_2)$ and  $(B_1,B_2)$ are in 
the same $\mathfrak T$-orbit.
\item If  $A_1$ and $B_1$ are in different $\mathfrak T_1$-orbits, then also  $(A_1,A_2)$ and $(B_1,B_2)$ are in different $\mathfrak T$ orbits.
\end{itemize}

\end{lemma}

\newpage

Now we have an effective algorithm:

\begin{itemize}
\item
Compute a set $\mathfrak R_1$ of representatives of the action $\mathfrak T_1$ on $\mathcal B(G,T_1)$. 
\item
Compute a set $\mathfrak R_2$ of representatives of the Hurwitz action on $\mathcal B(G,T_2)$. 
\item 
Determine the set of tuples  $(A_1,A_2) \in \mathfrak R_1 \times \mathfrak R_2 $ which satisfy:
\[
\Sigma(A_1) \cap \Sigma(A_2) =\lbrace 1_G \rbrace.
\]
This set is denoted by $\mathfrak R$.
\end{itemize}

We achieve the following: Every orbit of $\mathfrak T$ has at least one representative in $\mathfrak R$ by \ref{pen11}. Hence we have an upper bound 
for the number $n$ of $\mathfrak T$-orbits. We also have a lower bound for $n$. The pairs $(A_1,A_2)$ and $(B_1,B_2)$ are in different orbits of $\mathfrak T$, if 
\begin{enumerate}
\item
$A_1 \ne B_1$ or if  
\item
$A_2$ and $B_2$ are in different orbits of  $\mathfrak T_2$  (cf. \ref{pen11}). 
\end{enumerate}

\begin{itemize}
\item 
In most cases it is possible to determine the number $n$ of orbits of $\mathfrak T$ using this method. 
If this is not possible, then we  exchange the pairs $T_1$ and $T_2$ and compute the set $\mathfrak R$ again. 
If it is still not possible to determine $n$ we have to identify the pairs of the smaller set using the action $\mathfrak T$.
\end{itemize}

\bigskip

\section{The algorithm and the classification result}\label{section5}

\bigskip
\bigskip

In this section we explain our algorithm, which allows us to classify all regular surfaces isogenous to a product of curves 
of unmixed type with $\chi(\OH(S))=2$. For the implementation of the algorithm the computer algebra system MAGMA \cite{magma} is used. 
The program is based on the program in the appendix of \cite{B.P}. After the explanation of the algorithm,  we discuss the output of the computations. Finally we give our classification result. \\

\bigskip

Let $S= (C_1 \times C_2)/G$ be a regular surface isogenous to a product of unmixed type with $\chi(\OH_S) = 2$. 
From $2 |G| = (g(C_1) -1)(g(C_2) -1)$, $\alpha(T_1(S)) = g(C_2)-1$ and $\alpha(T_2(S)) = g(C_1)-1$ for the attached types $T_i(S)$, it follows 
\[
|G| = \frac{1}{2} \alpha(T_1)\alpha(T_2). 
\]
According to the list in lemma \ref{tupel},  $\alpha(T_i) \leq 168$ and therefore $|G| \leq 14112$. 
The group order is also bounded in terms of the genera $g(C_i)$, in fact \\
$|Aut(C_i)| \leq 84(g(C_i)-1)$ due to Hurwitz' famous theorem. 
For small $g(C_i)$ there are better bounds. In Breuer's book \cite[p.91]{Br00} there is a table
which gives the maximum order of $|Aut(C_i)|$ in case $2 \leq g(C_i) \leq 48$.

\bigskip

\begin{definition}
Let  $A=[m_1,...,m_r] $ be an  r-tuple of integers  $m_i \ge 2 $.
The  \emph{polygonal group} $\mathbb T(m_1,...,m_r)$ is defined as  
\[
\mathbb T(m_1,...,m_r)=<t_1,...,t_r \quad | \quad t_1\cdot ... \cdot t_r = t_1^{m_1}=...=t_r^{m_r}=1 >.
\]
\end{definition}

A group $G$ admitting a spherical system of generators of type $[m_1,...,m_r]$ is a quotient of $\mathbb T(m_1,...,m_r)$. 
The following lemma will be used in the sequel. The proof of it is elementary and will be omitted. 

\bigskip

\begin{lemma}\label{abeliani}

Let $G$ be a group and  $H$ a quotient of $G$, then:
\begin{itemize}
\item $H^{ab}$ is a quotient of  $G^{ab}$.
\item The commutator subgroup  $[H,H]$ is a quotient of  $[G,G]$.
\item If $G$ is a quotient of $\mathbb T(2,3,7)$, then $G$ is perfect. 
\end{itemize}
\end{lemma}

\bigskip

We can now describe the algorithm briefly. We perform the following steps:

\bigskip
{\bf Step 1:}
The program computes the set $\mathcal N$ of types given in lemma \ref{tupel}. For every integer  $m \leq 14112$ we compute the set of all triples of the form 
$(m,T_1,T_2)$ up to permutation of $T_1$ and $T_2$, where $T_1, T_2 \in \mathcal N$ and $m=\frac{1}{2} \alpha(T_1)\alpha(T_2)$. 

\bigskip
{\bf Step 2:}
For every triple $(m,T_1,T_2)$ the script computes $g(C_2)= \alpha(T_1(S)) + 1$ and $g(C_1)=\alpha(T_2(S)) + 1$.
If $2 \leq g(C_i) \leq 48$ for at least one $i$, we check if $m$ is 
less or equal to the maximum group order of the automorphism group $Aut(C_i)$ allowed by Breuer's table.

\bigskip
{\bf Step 3:}
For every triple $(m,T_1,T_2)$ passing this test,
the script searches the list of groups of order $m$ for a  group admitting a spherical system of 
generators of type $T_1$ and one of type $T_2$.

\bigskip
{\bf Step 4:}
For each triple $(m,T_1,T_2)$ and each group $G$ of order $m$, admitting a spherical system of 
generators of type $T_1$ and of type $T_2$, the script computes the number of orbits of the action  
$\mathfrak T$ on $\mathcal B(G,T_1) \times \mathcal B(G,T_2)$, using the method explained after lemma \ref{pen11}.

\newpage

In {\bf Step 3} we face two computational difficulties:
\begin{itemize} 
\item 
In  MAGMA's \emph{Database of Small Groups} all groups of order $|G| \leq 2000$ are contained, except the groups of order $|G| = 1024$ (which can not occur). 
In the case $2001 \leq |G| \leq 14112$, there is no MAGMA \emph{Database} containing all groups of these orders. 
Only the perfect groups with $|G| \leq 50000$ are contained in  MAGMA's \emph{Database of Perfect Groups}.   
\item Groups of order
\[
|G| \in \{1920, 1152, 768, 512, 384, 256\}
\]
can occur. Despite the fact, that we have access to all groups of these orders, it is not efficient to search through all of them for  
spherical systems of generators, because the number of these groups is too high.
\footnote{Indeed there are $12.059.590$ groups $G$, such that $|G| \in \{1920, 1152, 768, 512, 384, 256\}$.}
\end{itemize}
Due to these difficulties we split the program into two \emph{main routines} namely $Mainloop1$ and $Mainloop2$. 

\begin{itemize}
\item 
	The function $Mainloop1$ treats the cases  where 
	\[
	|G| \geq 2001 ~~ \makebox{or} ~~  |G| \in \{1920, 1152, 768, 512, 384, 256\}. 
	\]
	\begin{itemize}
	\item[i)]
	The case $m=|G| \geq 2001$.  For all triples  $(m,T_1,T_2)$, we search the \emph{Database of Perfect Groups} for  a perfect group of order 
	$m$, admitting a spherical system of generators of type $T_1$ and  $T_2$. If neither $T_1=[2,3,7]$ nor $T_2=[2,3,7]$ we can not yet 
	decide 	if there is a non-perfect group of order $m$, admitting a spherical system of generators of type $T_i$ (cf. \ref{abeliani}).
	The script then saves the triple in the file $exceptional.txt$.  We have to investigate these cases with theoretical arguments (see 
	subsection \ref{xtc2}), and we will show these do not occur.  

	\item[ii)]
	The case $|G| \in \{1920, 1152, 768, 512, 384, 256\}$. 
	Each group $G$, admitting a spherical system of generators  of type $T_1=[n_1,...,n_r]$ and of type $T_2=[m_1,...,m_s]$  is a quotient of 
	\[
	\mathbb T(T_1) := \mathbb T(n_1,...,n_r) \quad  \makebox{and} \quad \mathbb T(T_2) := \mathbb T(m_1,...,m_s).
	\]
	According to lemma \ref{abeliani} the group $G^{ab}$ is a quotient of 
	\[
	\mathbb T(n_1,...,n_r)^{ab} \quad \makebox{and} \quad  \mathbb T(m_1,...,m_s)^{ab}.
	\]
	The script first computes all possible abelianizations of $G$ from the types. 
	With this step we can exclude the groups which don't have the right abelianization. 
	Then we search through the remaining groups 
	if there is a group, admitting  spherical systems of generators of type $T_i$. 
	\end{itemize}

\item
The function $Mainloop2$ treats the case where 
\[
|G| \leq 2000 ~~ \makebox{and} ~~ |G| \notin \{1920, 1152, 768, 512, 384, 256\}.
\]
The output is written in the file $loop2.txt$.
Due to the high  use of memory, if one of the types is $[2^8]$, we split the computation into two parts:
\begin{itemize}
\item[i)]
With the command $Mainloop2(n_1,n_2,1)$ the script classifies, in the sense of above, 
all surfaces where $n_1 \leq |G| \leq n_2$ and $T_i \neq [2^8]$.  
\item[ii)]
With the command $Mainloop2(n_1,n_2,0)$ the script classifies all surfaces where $n_1 \leq |G| \leq n_2$ and $T_1 = [2^8]$. 
 
\end{itemize}
Also in this main routine, if $T_1=[2^8]$ we have to treat some cases separately. Our workstation has not enough memory to compute the set 
$\mathcal B(G,[2^8])$, if 
\[
|G| \in \lbrace 168, 96, 48 \rbrace. 
\] 
The occurring triples $(n,T_1,T_2)$, where $n$ is one of the group orders above are marked by the script as exceptional. We treat these cases   in subsection \ref{xtc1}. 
\end{itemize}

Before we discuss the exceptional cases of the output, we explain some notation from group theory, that will be used in this and in the next section.

\begin{itemize}
\item
We use the following MAGMA notation:  $\langle a,b \rangle$ denotes the group of order $a$  having number $b$ in the 
database of Small Groups \cite{magma}. 
\item
The group  $U(4,2) \le Gl(4,\mathbb F_2) $ is defined to be the subgroup of upper triangle matrices $A=(a_{ij})$  with $a_{ii}=1$, for all $1 \leq i \leq 4$.
\item
The group $G(128,36):=\langle 128,36 \rangle $ is given in a polycyclic presentation:
\[
\langle 128,36 \rangle =
\left \langle g_1, ... ,g_7 \left| \left. \begin{matrix} g_1^2=g_4, & g_2^2=g_5, & g_2^{g_1}=g_2g_3 \\ 
               g_3^{g_1}=g_3g_6, & g_3^{g_2}=g_3g_7, & g_4^{g_2}=g_4g_6 \\ 
               g_5^{g_1}=g_5g_7 &  &  \\           
\end{matrix} \right\rangle \right. \right.
\]
Here $g_{i}^{g_{j}}$ means $g_{j}^{-1}g_{i}g_{j}$.
The squares of the generators $g_1, ... ,g_7$, which are not mentioned in the presentation are equal to $1$. 
If  $g_{i}^{g_{j}}= g_i$, this relation is omitted in the presentation. 
\end{itemize}

\bigskip

\subsection{Exceptional cases for Mainloop1:}\label{xtc2}

\bigskip
\bigskip
The following cases are skipped by the program, because the group order is greater than $2000$, and 
saved in the file $exceptional.txt$:

\begin{longtable}{p{2cm}p{2cm}p{2cm}p{1cm}p{1cm}} 
	$|G|$ & $T_1$ & $T_2$    \\ \hline
	$4608$  & $[ 2, 3, 8 ]$  &  $ [ 2, 3, 8 ]$\\
  $3840$  &  $[ 2, 3, 8 ]$ &  $[ 2, 4, 5 ]$ \\
  $3456$  &  $[ 2, 3, 8 ]$ &  $[ 2, 3, 9 ]$ \\
  $3200$  &  $[ 2, 4, 5 ]$ &  $[ 2, 4, 5 ]$ \\
	$2880$  &  $[ 2, 3, 9 ]$ &  $[ 2, 4, 5 ]$ \\
	$2880$  &  $[ 2, 3, 8 ]$ &  $[ 2, 3, 10 ]$\\
	$2592$  &  $[ 2, 3, 9 ]$ &  $[ 2, 3, 9 ]$ \\
  $2400$  &  $[ 2, 3, 10 ]$ & $[ 2, 4, 5 ]$ \\
  $2304$  &  $[ 2, 3, 8 ]$  & $[ 2, 4, 6 ]$ \\
  $2304$  &  $[ 2, 3, 8 ]$  & $[ 3, 3, 4 ]$ \\
  $2304$  &  $[ 2, 3, 8 ]$  & $[ 2, 3, 12 ]$ \\
  $2160$	 & $[ 2, 3, 9 ]$ &  $[ 2, 3, 10 ]$ \\
\end{longtable}

\bigskip
Our aim is to show, that the cases above can not occur. 

\begin{proposition} 
There are no finite groups  $G$ admitting a disjoint pair of spherical systems of generators $(A_1,A_2)$ of type $(T_1,T_2)$, where 
$|G|$, $T_1$ and $T_2$ are in the table above. 
\end{proposition}

\begin{proof}
Our MAGMA code has already excluded the cases where $G$ is perfect. 
We only treat the case 
$|G|=4608$,  $T_1=[ 2, 3, 8 ]$  and  $T_2= [ 2, 3, 8 ]$.
The other cases can be excluded using similar methods.
The ideas we use are from \cite{bauer}. 
The group $G$ is a quotient of  $\mathbb T(2,3,8)$ 
and there is a surjective homomorphism $\phi^{ab}: \mathbb T(2,3,8)^{ab} \to G^{ab}$. Since $G$ is not perfect,
$G^{ab} \neq \lbrace 1_G \rbrace$.
Similarly, the commutator subgroup  $G'=[G,G]$ is also a quotient of  $\mathbb T(2,3,8)'=[\mathbb T(2,3,8),\mathbb T(2,3,8)]$. 
We have 
\[
\mathbb T(2,3,8)^{ab} \simeq \mathbb Z_2, ~~ \makebox{and} ~~  \mathbb T(2,3,8)' \simeq \mathbb T(3,3,4). 
\]
This implies $G^{ab} \simeq \mathbb Z_2 $, thus $|G'|=2304$ and the group $G'$ is a quotient of $\mathbb T(3,3,4)$. 
\[
\mathbb T(3,3,4)^{ab} \simeq \mathbb Z_3, ~~ \makebox{and} ~~ \mathbb T(3,3,4)' \simeq \mathbb T(4,4,4). 
\]
Since $|G'|=2^8 \cdot 3^2$ the group $G'$ is solvable, due to Burnsides theorem \cite{burn1}. We have $G'^{ab} \simeq \mathbb Z_3$, thus: $|G''|=768$ and  $G''$ is a quotient of $\mathbb T(4,4,4)$.
This is impossible according to the following computational fact, which can be verified with MAGMA. 
\bigskip
\end{proof}

\begin{lemma}\cite[Lemma 4.11]{bauer}
There are $1090235$ groups of order $768$. None of them is a quotient of $\mathbb T(4,4,4)$.
\end{lemma}

\bigskip
\subsection{Exceptional cases for Mainloop2:}\label{xtc1}

\bigskip
\bigskip

Here $T_1=[2^8]$ and $T_2 \in \mathcal N$. We have $|G|= \frac{1}{2} \alpha(T_1) \alpha(T_2) = \alpha(T_2) \leq 168$. Since our workstation has not enough memory to compute the set $\mathcal B(G,[2^8])$  if $|G| \in \lbrace 168, 96, 48 \rbrace$, 
these cases are marked  as exceptional. 
We receive the following output:
\vspace{0.5cm}

\begin{longtable}{p{2cm}p{2.5cm}p{3cm}p{3cm}p{3cm}p{3cm}}
$T_2$ & $G$ & No. of $\mathfrak T$-orbits  &   \\ \hline

$[2,3,7]$ & $\langle 168,42 \rangle$ & $0$ & exceptional case  \\
$[2,3,8]$ & $\langle 96,64 \rangle$ & $0$ & exceptional case  \\
$[2,4,6]$ & $\langle 48,48 \rangle$ & $2$  & exceptional case  \\
$[3,4^2]$ & $\langle 24,12 \rangle$ &  $1$  &    \\
$[2^3,4]$ & $\langle 16,11 \rangle$ & $2$  &     \\
$[2^5]$   & $\langle 8,5   \rangle$ & $1$  &     \\

\end{longtable}

Next, we will investigate the exceptional cases. 

\begin{proposition}
The group $\langle 168, 42 \rangle $ has no disjoint pair of spherical systems of generators of type $([2^8],[2,3,7])$.
\end{proposition}

\begin{proof}
A  MAGMA computation shows, that this group has only one conjugacy class of elements of order $2$. 
Hence, for every pair $(A_1,A_2)$ of generators of type $([2^8],[2,3,7])$ the intersection 
$\Sigma(A_1) \cap \Sigma(A_2)$ is nontrivial. 
\end{proof}

\begin{proposition}
The group $\langle 96,64 \rangle$ has no disjoint pair of spherical systems of generators of type  $([2^8],[2,3,8])$.
\end{proposition}
\begin{proof}
A  MAGMA calculation shows, that this group has two conjugacy classes of elements of order $2$. We denote them by $K_1$ and $K_2$. 
We have $|K_2|=3$  and $\langle K_2 \rangle \simeq \mathbb Z_2 \times \mathbb Z_2$.
Since  $|\langle K_2 \rangle | = 4 $, there is no spherical system of generators $A_1=(h_1, ..., h_8)$ of 
$G$ of type $[2^8]$, with $h_i \in K_2$ for all $1 \leq i \leq 8$. 
Every spherical system of generators of $G$ contains elements of $K_1$. 
A further MAGMA calculation shows, that there is no spherical system of generators $A_2=(g_1,g_2,g_3)$ of $G$ of type $[2,3,8]$, 
with $g_1\in K_2$. 
\end{proof}

\begin{proposition}
The group $\langle 48,48 \rangle$ admits disjoint pairs of spherical systems of generators of type  $([2^8],[2,4,6])$. The number of $\mathfrak T$-orbits is two.
\end{proposition}

\begin{proof}
The group $G$ has $19$ elements of order $2$, they are contained in $5$ conjugacy classes. We denote them by $K_1, ... ,K_5$ and the set of elements of order two by $M$.

\vspace{0.5cm}

\begin{longtable}{p{2cm}p{3,5cm}p{2.5cm}p{20cm}}
class & rep & length \\ \hline
$K_1$ & $G.2$     & $1$ \\
$K_2$ & $G.4$     & $3$ \\
$K_3$ & $G.2*G.4$  & $3$ \\
$K_4$ & $G.1$     & $6$ \\
$K_5$ & $G.1*G.2$ & $6$ \\
\end{longtable}

\vspace{0.5cm}

\noindent 
A  MAGMA calculation shows, that $g_1 \in K_4 \cup K_5 $ for each spherical system of generators  $(g_1,g_2,g_3) \in \mathcal B(G,[2,4,6])$. Since we are interested in disjoint pairs of spherical systems of generators, it is not necessary to compute the whole set 
$\mathcal B(G,[2^8])$. All elements  $(h_1, ... ,h_8) \in \mathcal B(G,[2^8])$, which contain some $h_i \in K_4$, 
as well as some $h_j \in K_5$ are irrelevant. For each $1\leq l \leq 5$ we define  subsets  
\[
\mathcal B_l:=\lbrace [h_1, ... ,h_8] \in \mathcal B(G,[2^8]) ~ \big\vert  ~ h_i \in M \setminus K_l \rbrace
\]
of $\mathcal B(G,[2^8])$ 
and compute them in the cases  $l=4$ and $l=5$.
We have  $|\mathcal B_4|=|\mathcal B_5|= 9.213.120$. Since the Hurwitz action acts via conjugation and permutation of elements, we can restrict it to $\mathcal B_4$ and to $\mathcal B_5$. 
Next we compute for each orbit of the restricted actions a representative. 
The sets of representatives are denoted by $\mathcal R_4$ and $\mathcal R_5$. We have $|\mathcal R_4|=|\mathcal R_5| =10$.
All elements in $\mathcal B(G,[2,4,6])$ are contained in two orbits of the  Hurwitz action. We denote by $A_1$ and $A_2$ 
two representatives of these orbits. 
There are two disjoint pairs in $\lbrace A_1, A_2 \rbrace \times \mathcal R_4$, and two disjoint pairs in 
$\lbrace A_1, A_2 \rbrace \times \mathcal R_5$.
According to \ref{pen11}, we have at least one representative for each  $\mathfrak T$-orbit. 
Using the function "Orbi" we can identify the pairs above, which are $\mathfrak T$-equivalent.  We find two equivalence classes. 
To verify the above computations a source code can be found at
\url{http://www.staff.uni-bayreuth.de/~ bt300503/} in the file $script.txt$.
\end{proof}

\bigskip

\begin{remark}\label{exchange}

\end{remark}

From the output files we can see that there is exactly one occurrence with $T_1(S)=T_2(S)$ and $n \geq2$. In this case 
$G= G(128,36)$. The types are $T_i=[4,4,4]$ and $n=2$. 
We denote by $(A_1,A_2)$ and $(B_1,B_2)$ the representatives for the $\mathfrak T$-orbits from the output file $loop2.txt$.
It remains to check if $(A_2,A_1)$ and $(B_1,B_2)$ are in the same $\mathfrak T$-orbit. 
A MAGMA computation shows that this is not the case. The source code for this computation is available at
\url{http://www.staff.uni-bayreuth.de/~ bt300503/} in the file $script.txt$.

\bigskip

Now we can give our main theorem, which implies in particular theorem \ref{families}.

\newpage

\begin{thm} \label{Haupterg}
Let $S=(C_1 \times C_2) / G$ 
be a regular surface isogenous to a product of unmixed type with $\chi(\OH_S)= 2$. Then 
$g(C_1)$, $g(C_2)$, the group $G$ and the corresponding types $T_1(S)$, $T_2(S)$ are:  
\begin{longtable}{p{1.2cm}p{1.2cm}p{3cm}p{2.5cm}p{2cm}p{2cm}p{1cm}}
$g(C_1)$ & $g(C_2)$ & $G$ & $Id$ & $T_1(S)$ & $T_2(S)$ & $n$ \\ \hline
\\
$17$ & $43$ & $PSL(2, \mathbb F_7) \times \mathbb Z_2$  & $\langle 336,209 \rangle$ & $[2,3,14]$ & $[4^3]$ & $2$ \\
$49$ & $9$ & $(\mathbb Z_2)^3 \rtimes _ \varphi \mathfrak S_4$ & $\langle192,955 \rangle$ & $[2^2,4^2]$ & $[2,4,6]$ & $2$ \\
$49$ & $8$  & $ PSL(2, \mathbb F_7) $ & $\langle 168,42 \rangle$ & $[7^3]$ & $[3^2,4]$ & $2$ \\
$17$ & $22$  & $ PSL(2, \mathbb F_7) $ & $\langle 168,42 \rangle$ & $[3^2,7]$ & $[4^3]$ & $2$  \\
$5$ & $81$  & $(\mathbb Z_2)^4 \rtimes _ \varphi \mathcal D_5$ & $\langle 160,234 \rangle$ & $[2,4,5]$ & $[4^4]$ & $5$ \\
$17$ & $17$  & $ G(128,36)$ & $\langle 128,36 \rangle$ & $[4^3]$ & $[4^3]$ & $2$  \\
$9$ & $31$ & $ \mathfrak  S_5 $ & $\langle 120,34 \rangle$ & $[2,5,6]$ & $[2^2,4^2]$ & $1$ \\
$5$ & $49$  & $(\mathbb Z_2)^4 \rtimes _ \varphi \mathcal D_3 $ & $\langle 96,195 \rangle$ & $[2,4,6]$ & $[4^4]$ & $1$  \\
$25$ & $9$  & $(\mathbb Z_2)^4 \rtimes _ \varphi \mathcal D_3$ & $\langle 96,227 \rangle$ & $[2^5]$ & $[3,4^2]$ & $1$  \\
$9$ & $17$  & $ (\mathbb Z_2)^3 \rtimes _ \varphi \mathcal D_4$ & $\langle 64,73 \rangle$ & $[2^3,4]$ & $[2^2,4^2]$ & $1$  \\
$9$ & $17$  & $ U(4,2)$ & $\langle 64,138 \rangle$ & $[2^3,4]$ & $[2^2,4^2]$ & $1$ \\
$13$ & $11$   & $\mathcal A_5$ & $\langle 60,5 \rangle$ & $[5^3]$ & $[2^2,3^2]$ & $1$ \\
$41$ & $4$  & $\mathcal A_5 $ & $\langle 60,5 \rangle$ & $[3^5]$ & $[2,5^2]$ & $2$ \\
$9$ & $16$  & $\mathcal A_5$ & $\langle 60,5 \rangle$ & $[3,5^2]$ & $[2^5]$ &  $2$  \\
$5$ & $31$  & $\mathcal A_5 $ & $\langle 60,5 \rangle$ & $[3^2,5]$ & $[2^6]$ & $1$  \\
$5$ & $25$  & $ \mathfrak  S_4 \times \mathbb Z_2 $ & $\langle 48,48 \rangle$ & $[2^3,3]$ & $[4^4]$ & $1$ \\
$9$ & $13$  & $ \mathfrak S_4 \times \mathbb Z_2 $ & $\langle 48,48 \rangle$ & $[2^3,6]$ & $[2^2,4^2]$ & $1$ \\
$13$ & $9$ & $\mathfrak  S_4 \times \mathbb Z_2 $ & $\langle 48,48 \rangle$ & $[2^5]$ & $[4^2,6]$ & $1$ \\
$3$ & $49$  & $\mathfrak  S_4 \times \mathbb Z_2 $ & $\langle 48,48 \rangle$ & $[2,4,6]$ & $[2^8]$ & $2$ \\
$9$ & $9$ & $ (\mathbb Z_2)^3 \rtimes _ \varphi \mathbb Z_4 $ & $\langle 32,22 \rangle$ & $[2^2,4^2]$ & $[2^2,4^2]$ & $1$ \\
$9$ & $9$  & $ \mathcal D_4 \times (\mathbb Z_2)^2$ & $\langle 32,46 \rangle$ & $[2^5]$ & $[2^5]$ & $1$ \\
$17$ & $5$ & $(\mathbb Z_2)^4 \rtimes _ \varphi \mathbb Z_2 $ & $\langle 32,27 \rangle$ & $[2^3,4^2]$ & $[2^3,4]$ & $3$  \\ 
$9$ & $9$  & $(\mathbb Z_2)^4 \rtimes _ \varphi \mathbb Z_2 $ & $\langle 32,27 \rangle$ & $[2^2,4^2]$ & $[2^5]$ & $1$  \\ 
$5$ & $13$  & $\mathfrak S_4  $ & $\langle 24,12 \rangle$  & $[2^2,3^2]$ & $[4^4]$ & $1$ \\
$3$ & $25$  & $\mathfrak S_4$ & $\langle 24,12 \rangle$  & $[3,4^2]$ & $[2^8]$ & $1$ \\
$5$ & $9$   & $\mathcal D_4 \times \mathbb Z_2$ & $\langle 16,11 \rangle$ & $[2^5]$ & $[2^3,4^2]$ & $1$ \\
$9$ & $5$  & $(\mathbb Z_2)^2 \rtimes _ \varphi \mathbb Z_4 $ & $\langle 16,3 \rangle$ & $[2^3,4^2]$ & $[2^2,4^2]$ & $2$ \\ 
$9$ & $5$  & $(\mathbb Z_2)^4 $ & $\langle 16,14 \rangle$ & $[2^6]$ & $[2^5]$ &  $2$ \\
$3$ & $17$  & $\mathcal D_4 \times \mathbb Z_2 $ & $\langle 16,11 \rangle$ & $[2^3,4]$ & $[2^8]$ & $2$  \\
$7$ & $4$  & $(\mathbb Z_3 )^2$ & $\langle 9,2 \rangle$ & $[3^5]$ & $[3^4]$ & $1$  \\
$5$ & $5$  & $(\mathbb Z_2)^3$ & $\langle 8,5 \rangle$ & $[2^6]$ & $[2^6]$ & $1$ \\
$3$ & $9$ & $(\mathbb Z_2)^3 $ & $\langle 8,5 \rangle$ & $[2^5]$ & $[2^8]$ & $1$ \\
\end{longtable}
\noindent
Each row in the table corresponds to a union of connected components of the \emph{Gieseker moduli space} of surfaces of general type with  
$K_S^2=16$ and $=2$. The number $n$ of these components is given in the last column. 
\end{thm}

\bigskip
\bigskip

\noindent \textbf{Authors Address}:\\
Christian Glei\ss ner: Universit\"at Bayreuth, Lehrstuhl Mathematik VIII;\\
Universit\"atsstra\ss e 30, D-95447 Bayreuth, Germany\\
E-mail address: \url{christian.gleissner@uni-bayreuth.de}

\end{document}